\documentclass[11pt]{article}

\usepackage[T1]{fontenc}
\usepackage{amsmath,amssymb,amsthm}
\usepackage[margin=1in]{geometry}

\newtheorem{theorem}{Theorem}
\newtheorem{conjecture}{Conjecture}
\newtheorem{cor}{Corollary}
\newtheorem{lemma}{Lemma}

\newcommand{\F}{\mathcal F}
\newcommand{\Gr}[2]{\genfrac{[}{]}{0pt}{}{#1}{#2}_q}
\newcommand{\La}{\operatorname{La}}
\newcommand{\Laq}{\operatorname{La}_q}
\newcommand{\Sigmaq}{\Sigma_q}

\title{Forbidden subposet problems in the linear lattice}
\author{Bal\'azs Patk\'os \and Casey Tompkins}
\date{}

\begin{document}

\maketitle

\begin{abstract}
    We study weak and strong forbidden subposet problems in the linear lattice \(L_n(q)\). We show that the \(q\)-analogues of the Bukh--Griggs--Lu conjecture fail for every prime power \(q\). If \(q\ge3\) and \(n\) and \(d\) have opposite parity, then
\(
\La_q(n,L_d(q))=\La_q^*(n,L_d(q))=\Sigma_q(n,d).
\)
We also show that
\(\La_q(n,D_s)=\La_q^*(n,D_s)=\Sigma_q(n,2)
\)
for \(2\le s\le q\) and all \(n\), and for \(s=q+1\) when \(n\) is odd. For \(n\) even, a maximum strong \(D_{q+1}\)-free family can be contained from the three middle layers. The case $s=2$ settles the $q$-analog of the diamond conjecture in the affirmative.
\end{abstract}
\section{Introduction}

This paper addresses the following extremal problem for finite posets: if $(P,\le_P)$ and $(Q,\le_Q)$ are posets, then $Q'\subseteq Q$ is a \textit{weak copy} of $(P,\le_P)$ in $(Q,\le_Q)$ if there exists an order-preserving bijection $\iota:P\rightarrow Q'$. If $\iota$ is a poset-isomorphism ($p\le_P p'$ if and only if $\iota(p)\le_Q \iota(p')$), then $Q'$ is a \textit{strong copy} of $(P,\le_P)$. What is the maximum size $\La(P,Q)$ of a subset $R\subset Q$ that contains no weak copies of $P$ and similarly, what is the maximum size $\La^*(P,Q)$ of a subset $R\subset Q$ that contains no strong copies of $P$? If $P=C_2$ the chain on two elements, then $\La(C_2,Q)=\La^*(C_2,Q)$ is the maximum size of an antichain in $Q$. When $(Q,\le_Q)$ is the Boolean lattice $B_n=(2^{[n]},\subseteq)$, then we use the notation $\La(n,P)$ and $\La^*(n,P)$ introduced by Katona and Tarj\'an \cite{KatonaTarjan}. In the last 4 decades, there has been a considerable interest and research on this special case of forbidden subposet problems (see the surveys \cite{AMP,GLsurv} and Chapter 7 of \cite{GP}). Ellis, Ivan, and Leader \cite{EIL} obtained a breakthrough result by a construction that refuted a long-standing conjecture of the field. Several follow-up papers \cite{Patkos,PT,Tompkins} have appeared recently, and the aim of the present paper is to find similarities and differences between the Boolean case and the less extensively studied \cite{Gerbner,SS,SY,XT} case of $Q$ being the linear lattice $L_n(q)$. We write $\Laq(n,P):=\La(P,L_n(q))$ and $\Laq^*(n,P)=\La^*(P,L_n(q))$.

Let us introduce some notation. For a prime power $q$, let $\mathbb F_q$ denote the field of $q$ elements, and let $V$ be a vector space of dimension $n$ over $\mathbb F_q$, most often we think of $V$ as $\mathbb F_q^n$.  $L_n(q)$ is then the set of all subspaces of $V$ ordered by inclusion.
The set $\mathcal L_d=\Gr{V}{d}$ of all $d$-dimensional subspaces of $V$ has size $\Gr{n}{d}=\prod_{i=0}^{d-1}\frac{q^{n-i}-1}{q^{d-i}-1}$ and by symmetry
$\Gr nr=\Gr n{n-r}$. We will sometimes write  $[r]_q=\Gr r1=\frac{q^r-1}{q-1}$ so that $\Gr{n}{d}=\frac{[n]_q\cdot [n-1]_q\cdot \ldots \cdot [n-d+1]_q}{[d]_q\cdot [d-1]_q\cdot \ldots \cdot [1]_q}$.
We denote by $\Sigma(n,k)=\sum_{i=1}^k\binom{n}{\lfloor \frac{n-k}{2}\rfloor+i}$ and $\Sigma_q(n,k)=\sum_{i=1}^k\Gr{n}{\lfloor \frac{n-k}{2}\rfloor+i}$ the sum of the $k$ largest binomial and $q$-nomial coefficients.
 
There is a natural and simple way to construct large $P$-free families in $B_n$ and in $L_n(q)$. Let $e(P)$ be the maximum integer $k$ such that any consecutive set of $k$ layers is weak $P$-free in $B_n$ for any $n$. We write $e^*(P)$, $e_q(P)$, and $e^*_q(P)$ for the strong and vector space variants. By definition, the inequalities $\La(n,P)\ge \Sigma(n,e(P))$, $\La^*(n,P)\ge \Sigma(n,e^*(P))$, $\Laq(n,P)\ge \Sigma_q(n,e_q(P))$, and $\Laq^*(n,P)\ge \Sigma_q(n,e^*_q(P))$ hold. It was conjectured in \cite{Bukh,GriggsLu} that asymptotic equality holds for all posets $P$ in the first inequality, and later analogous conjectures were made \cite{Gerbner,P2} for the other cases.

As mentioned before, the conjectures for the Boolean case were refuted by Ellis, Ivan, and Leader \cite{EIL}, and in \cite{PT}, the authors obtained the stronger result that for any integer $K$ there exists a poset $P$ with $e(P)=2$ and $\La(n,P)\ge K\binom{n}{\lfloor \frac{n}{2}\rfloor}$. This latter statement cannot have an analog in $L_n(q)$ simply because it is more concentrated on its middle levels, and $|L_n(q)|=\sum_{i=0}^n\Gr{n}{i}\le 3\Gr{n}{\lfloor \frac{n}{2}\rfloor}$ for any $n$ and $q$.

\bigskip

Our first result, a construction, disproves the $q$-analog conjectures. Most constructions showing $\La(n,P)>(1+\varepsilon)\Sigma(n,e(P))$ use more than $e(P)$ middle layers in a way that missing parts of 'very middle layers' are overcompensated by lowest and highest layers of the construction. This is possible as $\binom{n}{\lfloor \frac{n}{2}\rfloor+i}$ are asymptotically equal provided $|i|<C$ for some constant $C$. This is not true for $q$-nomial coefficients, and so these types of constructions do not seem to have a $q$-analog. On the other hand, as it was pointed out in \cite{GerbnerPatkos}, the construction of \cite{EIL} can be used to obtain constructions for special posets of height\footnote{the height $h(P)$ of a poset $P$ is the maximum number of elements forming a chain in $P$} 2. For a graph $G$ its \textit{vertex--edge incidence poset} $I_G$ is $(V(G)\cup E(G)),\subseteq)$, i.e. elements are vertices and edges of $G$ and $v<e$ if and only if the vertex $v$ is incident to the edge $e$.
Let us write $P_q=I_{K_{q+2}}$.

Note that for any poset $P$, we have $e_q(P)\le e(P), e_q^*(P)\le e^*(P)$ as $B_n$ embeds into $L_n(q)$ for any prime power $q$ as shown by the mapping $\iota(A)=\langle e_a:a\in A\rangle$ where $e_1,e_2,\dots,e_n$ is an arbitrary basis of $\mathbb F_q^n$. Also, for any graph $G$ with at least one edge, we have $e^*(I_G)=1$. Indeed, if $n>2|V(G)|$ with $V(G)=\{v_1,v_2,\dots,v_\ell\}$, then the mapping defined as $\iota(v_i)=\{i\}\cup K$ and $\iota(e)=\{i:v_i\in e\}\cup K$ with $K=[\lfloor \frac{n}{2}\rfloor +1,n]$ shows $e^*(I_G)<2$. In particular, $e_q(P_q)=e^*_q(P_q)=1$.

\begin{theorem}\label{conjfalse}
    For any prime power $q$, we have \[
    \frac{\Laq^*(n,P_q)}{\Gr{n}{\lfloor \frac{n}{2}\rfloor}}\ge \frac{\Laq(n,P_q)}{\Gr{n}{\lfloor \frac{n}{2}\rfloor}}\ge \begin{cases}
  1+\dfrac{c_q}{q}+o(1),& n\text{ even},\\[7pt]
  1+c_q+o(1),& n\text{ odd},
 \end{cases}
 \qquad
 c_q:=\displaystyle\prod_{i=1}^{\infty}(1-q^{-i}).\]
\end{theorem}

The paper \cite{EIL} showed that for any $B_d$ with $d\ge 4$, we have $\La(n,B_d)\ge (1+\varepsilon_d)\Sigma(n,d)$ (and the same statement for the cases $d=2$ and $d=3$ was obtained in \cite{Tompkins} and \cite{Patkos}, respectively), i.e. small Boolean posets are counterexamples to the old conjecture in the Boolean case. We believe that this is not the case for small linear lattices in the linear lattice case. Exact equality with the sum of the largest layers can fail.
Nevertheless, we conjecture that these layers are asymptotically
optimal when the field and the forbidden subspace lattice are fixed.

\begin{conjecture}\label{conj}
For every fixed prime power $q$ and every fixed integer $d\ge 2$, if $n$ tends to infinity, then
\[
\operatorname{La}_q^*(n,L_d(q))
  = (1+o(1))\Sigma_q(n,d).
\]
\end{conjecture}

We establish the stronger conclusion of exact equality when
$q\ge 3$ and $n$ and $d$ have opposite parity.
 This parity assumption  makes the $d$ largest
layers a unique symmetric block of consecutive layers such that all outside layers are much smaller than those in the block. If $n$ and $d$ are of the same parity, then the smallest layer in the block and the largest outside the block have the same size.

\begin{theorem}
\label{thm:subspace-lattice-opposite-parity}
Let $d\geq2$ and $n\geq d+1$, and suppose that $n$ and $d$ have
opposite parity.  If $q\geq3$ is a prime power, then $\Laq^*(n,L_d(q))=\Laq(n,L_d(q))=\Sigmaq(n,d)$.
\end{theorem}

The proof of Theorem \ref{thm:subspace-lattice-opposite-parity} is based on two lemmas: one uses a double counting argument to show a lower bound on the sum of the missing normalized parts of any $d+1$ layers of a strong $L_d(q)$-free family, while the other is an estimate that compares this missing part to the size of middle layers (and which holds only in the opposite parity case). 

Theorem \ref{thm:subspace-lattice-opposite-parity} has an easy but important corollary. Suppose $P$ contains both a smallest and a largest element. Then $e_q(P)=\min\{d: L_d(q)$ contains a weak copy of $P\}$ and $e^*_q(P)=\min\{d: L_d(q)$ contains a strong copy of $P\}$. So Theorem \ref{thm:subspace-lattice-opposite-parity} immediately implies the following.

\begin{cor}
Suppose $P$ contains both a smallest and a largest element and $q\ge 3$ is a prime power. 
\begin{enumerate}
    \item 
    If $n-e_q(P)$ is even, then $\Laq(n,P)=\Sigma_q(n,e_q(P))$.
    \item 
    If $n-e^*_q(P)$ is even, then $\Laq^*(n,P)=\Sigma_q(n,e^*_q(P))$.
\end{enumerate}
\end{cor}

\smallskip

Then we give another proof of Theorem \ref{thm:subspace-lattice-opposite-parity} in the special case of $d=2$ that 'almost works' even if $n$ is even and also covers the $q=2$ case. For an integer $s\geq2$, let $D_s$ be the \textit{generalized diamond}: it has
a minimum element $a$, a maximum element $c$, and $s$ distinct middle
elements $b_1,\ldots,b_s$, with
 $a<b_i<c$ for all $1\leq i\leq s$. Note that $L_2(q)=D_{q+1}$.

\begin{theorem}
\label{thm:diamond-positive}
Let $q$ be a prime power and let $n\geq1$.
Then $\Laq^*(n,D_s)=\Laq(n,D_s)=\Sigmaq(n,2)$ in each of the following cases:
\begin{enumerate}
 \item $2\leq s\leq q$, for every $n$;
 \item $s=q+1$, when $n$ is odd.
\end{enumerate}
Moreover, if $n=2m$ and $s=q+1$, then there is a maximum strong
$D_{q+1}$-free family contained in
 $\mathcal L_{m-1}\cup\mathcal L_m\cup\mathcal L_{m+1}$.
\end{theorem}

Note that the case $s=2$ of the diamond poset $\Diamond=D_2$ is covered by Theorem \ref{thm:diamond-positive} for any value of $q$ and $n$. This result was first conjectured by Sarkis and Shahriari in \cite{SS}.

The proof of Theorem~\ref{thm:diamond-positive} is in the spirit of Sperner's original proof of the upper bound for the maximum size of an antichain in $B_n$: it starts with a maximum size $D_s$-free family $\F$, and as long as $\F$ contains subspaces of too small (large) dimension, one applies a Hall condition argument to replace members of $\mathcal L_k\cap \F$ with members of $(\mathcal L_{k+1}\cup \mathcal L_{k+2})\setminus \F$ where $k$ is the minimum dimension a member of $\F$ has.

\section{A two-level $P_q$-free family}

Fix a subspace $W\leq V$ of dimension
$n-\lfloor \frac{n}{2}\rfloor-1$.
Consider the family
\[
 \F=
 \Gr{V}{\lfloor \frac{n}{2}\rfloor}
 \;\cup\;
 \left\{U\in \Gr{V}{\lfloor \frac{n}{2}\rfloor+1}:U\cap W=\{0\}\right\}.
\]
The next theorem immediately implies Theorem~\ref{conjfalse}.

\begin{theorem}\label{twolayercounter}
For $n\geq 2$, the family $\F$ contains no weak copy of $P_q$.  Moreover,
\[
 |\F|=\Gr{n}{\lfloor \frac{n}{2}\rfloor}+q^{(\lfloor \frac{n}{2}\rfloor+1)(n-\lfloor \frac{n}{2}\rfloor-1)}.
\]
Consequently, if $q$ is fixed and $n\to\infty$, then
\[
 \frac{|\F|}{\Gr{n}{\lfloor \frac{n}{2}\rfloor}}
 =
 \begin{cases}
  1+\dfrac{c_q}{q}+o(1),& n\text{ even},\\[7pt]
  1+c_q+o(1),& n\text{ odd},
 \end{cases}
 \qquad
 c_q:=\displaystyle\prod_{i=1}^{\infty}(1-q^{-i}).
\]
\end{theorem}

We first state a lemma that states in a somewhat technical way that if three $d$-subspaces have pairwise intersections all of dimension $d-1$, then either all of them are contained in a $(d+1)$-space or they contain the same $(d-1)$-space.

\begin{lemma}\label{lem:star}
Let $A,X,Y$ be distinct $\lfloor \frac{n}{2}\rfloor$-dimensional subspaces such that every two
of them have $(\lfloor \frac{n}{2}\rfloor-1)$-dimensional intersection.  If
$ A+X\neq A+Y$,
then $A\cap X=A\cap Y$.
\end{lemma}

\begin{proof}
The assertion is immediate when $\lfloor \frac{n}{2}\rfloor=1$, since both intersections are the
zero subspace.  Suppose that $\lfloor \frac{n}{2}\rfloor\geq2$, and set
$S_X=A\cap X$, $S_Y=A\cap Y$.
Assume for a contradiction that $S_X\neq S_Y$.  These are distinct
hyperplanes of $A$, so
$\dim(S_X\cap S_Y)=\lfloor \frac{n}{2}\rfloor-2$.

Let $\rho:V\to V/A$ be the quotient map.  Since $\dim(X\cap A)=\dim(Y\cap A)=\lfloor \frac{n}{2}\rfloor-1$, both $\rho(X)$ and $\rho(Y)$ are one-dimensional.
Furthermore, $A+X\neq A+Y$ if and only if  $\rho(X)\neq\rho(Y)$.
Thus $\rho(X)$ and $\rho(Y)$ are distinct one-dimensional subspaces,
and hence have zero intersection.  If $z\in X\cap Y$, then
$\rho(z)\in\rho(X)\cap\rho(Y)=\{0\}$, so $z\in A$.  It follows that
\[
 X\cap Y
 \subseteq A\cap X\cap Y
 =S_X\cap S_Y.
\]
The reverse inclusion is automatic, and therefore $\dim(X\cap Y)=\lfloor \frac{n}{2}\rfloor-2$,
contrary to the hypothesis that $\dim(X\cap Y)=\lfloor \frac{n}{2}\rfloor-1$.  Hence
$S_X=S_Y$.
\end{proof}

\begin{proof}[Proof of Theorem~\ref{twolayercounter}]
Suppose, for a contradiction, that $\F$ contains a weak copy of $P_q$.
Denote the images of the $q+2$ vertex-elements by $A_1,\ldots,A_{q+2}$
and the image of the edge $ij$ by $B_{ij}$.
Every vertex-element has an element strictly above it in $P_q$.
Since $\F$ occupies only dimensions $\lfloor \frac{n}{2}\rfloor$ and $\lfloor \frac{n}{2}\rfloor+1$, each $A_i$ must
therefore have dimension $\lfloor \frac{n}{2}\rfloor$, and every $B_{ij}$ must have dimension
$\lfloor \frac{n}{2}\rfloor+1$. For each $i<j$ we have $A_i,A_j\subset B_{ij}$.
The spaces $A_i$ and $A_j$ are distinct $\lfloor \frac{n}{2}\rfloor$-spaces contained in a
$(\lfloor \frac{n}{2}\rfloor+1)$-space, and hence $\dim(A_i\cap A_j)=\lfloor \frac{n}{2}\rfloor-1$, and $B_{ij}=A_i+A_j$.

The elements $B_{1i}$, $2\leq i\leq q+2$, represent distinct edges of
$K_{q+2}$, so they are distinct subspaces.  Thus, for distinct
$i,j\in\{2,\ldots,q+2\}$,
$A_1+A_i=B_{1i}\neq B_{1j}=A_1+A_j$.
Lemma~\ref{lem:star}, applied to $A_1,A_i,A_j$, shows that all the
spaces $A_1\cap A_i$ are equal.  Consequently, there is a fixed
$(\lfloor \frac{n}{2}\rfloor-1)$-dimensional subspace $S$ such that
$S\subset A_i$ for all $1\leq i\leq q+2$.

As every $B_{ij}\in \F$ must have dimension $\lfloor \frac{n}{2}\rfloor+1$, therefore
 $B_{ij}\cap W=\{0\}$.
Since $S\subset B_{ij}$, this also gives $S\cap W=\{0\}$.
Pass to the quotient space $X=V/S$
and let $\overline W=(W+S)/S\leq X$.
Because $S\cap W=\{0\}$, we have
 $\dim X=n-\lfloor \frac{n}{2}\rfloor+1$,
and $\dim\overline W=\dim W=n-\lfloor \frac{n}{2}\rfloor-1$. Therefore
$\dim(X/\overline W)=2$.

For every $i$, the quotient $A_i/S$ is a one-dimensional subspace of
$X$.  For every pair $i<j$,
 $B_{ij}/S=(A_i/S)+(A_j/S)$
is two-dimensional.  We claim that
\[
 (B_{ij}/S)\cap\overline W=\{0\}.
\]
Indeed, if $b+S=w+S$ for some $b\in B_{ij}$ and $w\in W$, then
$b-w\in S\subset B_{ij}$, so $w\in B_{ij}\cap W=\{0\}$.  Thus the
common coset is the zero coset, proving the claim.

Let $\pi:X\rightarrow X/\overline W$
be the quotient map.  The preceding claim says that $\pi$ is injective
on every two-space $B_{ij}/S$.  It follows  that each
$\pi(A_i/S)$ is nonzero, and  that, for $i\neq j$, the two
one-dimensional spaces $\pi(A_i/S)$ and $\pi(A_j/S)$ are distinct:
otherwise the restriction of $\pi$ to their two-dimensional span
$B_{ij}/S$ could not be injective.  We have therefore obtained $q+2$
distinct one-dimensional subspaces of the two-dimensional vector space
$X/\overline W$. This is impossible, because a two-dimensional vector space over
$\mathbb F_q$ has exactly
$q+1$
one-dimensional subspaces.  This contradiction proves that $\F$ is
$P_q$-free.

\smallskip

It remains to count $\F$.  Its entire $\lfloor \frac{n}{2}\rfloor$-dimensional level contributes
$\Gr{n}{\lfloor \frac{n}{2}\rfloor}$.  The dimensions of $W$ and a member $U$ of the second
part add up to $n$.
Consequently, the selected $(\lfloor \frac{n}{2}\rfloor+1)$-spaces are exactly the complements of
$W$ in $V$.

Fix one complement $C$ of $W$, so
 $V=C\oplus W$ and $\dim C=\lfloor \frac{n}{2}\rfloor+1$.
For each linear map $f:C\to W$, define its graph by
\[
 \Gamma_f=\{c+f(c):c\in C\}\leq V.
\]
The map $c\mapsto c+f(c)$ is injective, so $\Gamma_f$ has dimension
$\lfloor \frac{n}{2}\rfloor+1$.  Also, if $c+f(c)\in W$, then the uniqueness of the decomposition
in $V=C\oplus W$ forces $c=0$, and hence $c+f(c)=0$.  Thus
$\Gamma_f\cap W=\{0\}$, and $\Gamma_f$ is a complement of $W$.

Conversely, let $U$ be any complement of $W$.  Let $p_C:V=C\oplus W\rightarrow C$
be the projection onto $C$ along $W$.  Its restriction
$p_C|_U:U\to C$ is injective because its kernel is $U\cap W=\{0\}$;
since $U$ and $C$ have the same dimension, it is an isomorphism.
Hence for every $c\in C$ there is a unique vector $u_c\in U$ whose
$C$-component is $c$.  Write it uniquely as
\[
 u_c=c+f_U(c),
 \qquad f_U(c)\in W.
\]
The map $f_U:C\to W$ is linear, because $U$ is a subspace and the two
coordinate projections are linear. By construction $U=\Gamma_{f_U}$.
This proves a bijection between the set of linear maps from $C$ to $W$ and the set of subspaces $U$ of $V$ with $V=U\oplus W$.

A linear map from a $(\lfloor \frac{n}{2}\rfloor+1)$-dimensional space to an
$(n-\lfloor \frac{n}{2}\rfloor-1)$-dimensional space is specified by
$(\lfloor \frac{n}{2}\rfloor+1)(n-\lfloor \frac{n}{2}\rfloor-1)$ independent field entries.  There are therefore
$q^{(\lfloor \frac{n}{2}\rfloor+1)(n-\lfloor \frac{n}{2}\rfloor-1)}$
such maps, and hence the same number of complements.  This proves
\[
 |\F|=\Gr{n}{\lfloor \frac{n}{2}\rfloor}+q^{(\lfloor \frac{n}{2}\rfloor+1)(n-\lfloor \frac{n}{2}\rfloor-1)}.
\]

Finally, the product formula for the Gaussian binomial coefficient gives
\[
 \Gr{n}{\lfloor \frac{n}{2}\rfloor}
 =q^{\lfloor \frac{n}{2}\rfloor(n-\lfloor \frac{n}{2}\rfloor)}
 \frac{\displaystyle\prod_{j=n-\lfloor \frac{n}{2}\rfloor+1}^{n}(1-q^{-j})}
      {\displaystyle\prod_{i=1}^{\lfloor \frac{n}{2}\rfloor}(1-q^{-i})}.
\]
Since $(\lfloor \frac{n}{2}\rfloor+1)(n-\lfloor \frac{n}{2}\rfloor-1)-\lfloor \frac{n}{2}\rfloor(n-\lfloor \frac{n}{2}\rfloor)=n-2\lfloor \frac{n}{2}\rfloor-1$,
we obtain the exact identity
\[
 \frac{|\F|}{\Gr{n}{\lfloor \frac{n}{2}\rfloor}}
 =1+q^{n-2\lfloor \frac{n}{2}\rfloor-1}
 \frac{\displaystyle\prod_{i=1}^{\lfloor \frac{n}{2}\rfloor}(1-q^{-i})}
      {\displaystyle\prod_{j=n-\lfloor \frac{n}{2}\rfloor+1}^{n}(1-q^{-j})}.
\]
For fixed $q$, as $n\to\infty$, the numerator product tends to
$c_q$ and the denominator product tends to $1$.  Moreover,
\[
 n-2\lfloor \frac{n}{2}\rfloor-1=
 \begin{cases}
  -1,&n\text{ even},\\
  0,&n\text{ odd}.
 \end{cases}
\]
\end{proof}

\section{Forbidding a smaller subspace lattice}

In this section, we prove Theorem \ref{thm:diamond-positive} and Theorem \ref{thm:subspace-lattice-opposite-parity}. We start with the proof of the former. The  core of the proof is the following lemma, which allows for a 'push-to-the-middle' argument.

\begin{lemma}
\label{lem:up-compression}
Let $2\leq s\leq q+1$, let $\F\subseteq L_n(q)$ be a nonempty strong
$D_s$-free family, let $d$ be the minimum dimension of a member of
$\F$, and put $t=q+2-s$.
Suppose that
\begin{equation}
 d\leq\frac{n-2}{2}
 \qquad\text{and}\qquad
 (q+1)[d+1]_q\leq t[n-d]_q.
 \label{eq:compression-hypotheses}
\end{equation}
Then there is a strong $D_s$-free family $\F'$ such that $|\F'|=|\F|$ and $\min\{\dim U:U\in\F'\}\geq d+1$.
\end{lemma}

\begin{proof}
Put $\F_d=\F\cap\mathcal L_d$.  For each $A\in\F_d$, define two
sets of possible replacements:
\begin{align*}
 X_A={}&\{B\in\mathcal L_{d+1}:A<B,\ B\notin\F\},\\
 Y_A={}&\bigl\{C\in\mathcal L_{d+2}:A<C\text{ and there are distinct }
 B_1,\ldots,B_s\in\F\cap\mathcal L_{d+1}\\
 &\hspace{43mm}\text{such that }A<B_i<C\text{ for every }i\bigr\}.
\end{align*}
Every member of $X_A$ is absent from $\F$ by definition.  Every member
$C$ of $Y_A$ is absent as well: otherwise
$A,B_1,\ldots,B_s,C$ would form a copy of $D_s$ in $\F$.

Make a bipartite graph whose left vertex class is $\F_d$, whose right
vertex class is
\[
 \bigcup_{A\in\F_d}(X_A\cup Y_A),
\]
and in which $A$ is adjacent precisely to the members of
$X_A\cup Y_A$.  We verify Hall's condition.  Fix
$\mathcal A\subseteq\F_d$, and write
\[
a=|\mathcal A|,
 \quad
 X=\bigcup_{A\in\mathcal A}X_A,
 \quad x=|X|,
 \quad
 Y=\bigcup_{A\in\mathcal A}Y_A,
 \quad y=|Y|.
\]
Note that $X$ and $Y$ are disjoint because they lie in different
layers.  We must prove $a\leq x+y$.

Count triples $A<B<C$ satisfying
$ A\in\mathcal A$, $B\in\F\cap\mathcal L_{d+1}$, $C\in\mathcal L_{d+2}$,
and denote their number by $T$.  For every $A\in\mathcal A$ there are
$[n-d]_q$ members of $\mathcal L_{d+1}$ containing $A$.  If such a
member is missing from $\F$, it belongs to $X_A$.  A fixed
$B\in X$ contains at most
$[d+1]_q$ members of $\mathcal A$.  Hence the number of pairs $(A,B)$ with
$A\in\mathcal A$, $B\in\F\cap\mathcal L_{d+1}$, and $A<B$ is at
least $a[n-d]_q-x[d+1]_q$.
Each such pair extends to exactly $[n-d-1]_q$ choices for $C$, so
\begin{equation}
 T\geq
 \bigl(a[n-d]_q-x[d+1]_q\bigr)[n-d-1]_q.
 \label{eq:T-lower}
\end{equation}

For a fixed pair $A<C$ with dimensions $d$ and $d+2$, respectively,
the interval $[A,C]$ contains exactly
$q+1=(s-1)+t$
subspaces of dimension $d+1$.  If $C\notin Y$, at most $s-1$ of
these subspaces belong to $\F$; otherwise $C$ would lie in $Y_A$.
If $C\in Y$, there are at most $q+1$ such subspaces, so the possible excess over
$s-1$ is at most $t$.  There are $\Gr{n-d}{2}$ choices for $C$ above a fixed
$A$, while a fixed $C\in\mathcal L_{d+2}$ contains at most $\Gr{d+2}{d}$
members of $\mathcal A$.  Therefore
\begin{equation}
 T\leq (s-1)a\Gr{n-d}{2}+t\Gr{d+2}{2}y.
 \label{eq:T-upper}
\end{equation}

Since $[n-d]_q[n-d-1]_q=(q+1)\Gr{n-d}{2}$,
combining \eqref{eq:T-lower} and \eqref{eq:T-upper} gives
\begin{equation}
 ta\Gr{n-d}{2}\leq [d+1]_q[n-d-1]_q x+t\Gr{d+2}{2}y.
 \label{eq:hall-intermediate}
\end{equation}
The first condition in \eqref{eq:compression-hypotheses} gives
$d+2\leq n-d$, and hence $\Gr{d+2}{2}\leq \Gr{n-d}{2}$.  The second condition in \eqref{eq:compression-hypotheses}, together
with $\Gr{n-d}{2}=\frac{[n-d]_q[n-d-1]_q}{q+1}$,
gives
\[
 [d+1]_q[n-d-1]_q\leq t\Gr{n-d}{2}.
\]
Thus \eqref{eq:hall-intermediate} implies
\[
 ta\Gr{n-d}{2}\leq t\Gr{n-d}{2}(x+y),
\]
and consequently $a\leq x+y$.  Hall's theorem now supplies an
injection $\phi$ from $\F_d$ to
$\bigcup_{A\in\F_d}(X_A\cup Y_A)$
such that $\phi(A)\in X_A\cup Y_A$ for every $A\in\F_d$.  Define
 $\F'=(\F\setminus\F_d)\cup\phi(\F_d)$.
All the images are new and $\phi$ is injective, so $|\F'|=|\F|$;
also every member of $\F'$ has dimension at least $d+1$.

It remains to check that the simultaneous replacement has not created
a generalized diamond.  Suppose that
 $L<M_1,\ldots,M_s<U$
is a copy of $D_s$ in $\F'$.  Since the minimum dimension in $\F'$ is
at least $d+1$, we have $\dim U\geq d+3$.  Every new member has
dimension $d+1$ or $d+2$, so $U$ is an old member of $\F$.

A new member of dimension $d+1$ cannot be one of the $M_i$, because
there is no member of $\F'$ of smaller dimension to serve as $L$.
Suppose that a new member $C=\phi(A)\in Y_A$ of dimension $d+2$ is
one of the $M_i$.  By the definition of $Y_A$, there are distinct
$B_1,\ldots,B_s\in\F\cap\mathcal L_{d+1}$ with $A<B_j<C<U$ for all $1\leq j\leq s$.
Then $A,B_1,\ldots,B_s,U$ is a strong copy of $D_s$ in the original family
$\F$, a contradiction.  Thus no new member is a middle element, and
we already know that no new member is the top element.  If the alleged
copy used no new member, it would already lie in $\F$.  Otherwise its
only new member is the bottom, say $L=\phi(A)$.  Since
$A<L<M_i<U$ for every $i$, replacing $L$ by $A$ again produces a strong
copy of $D_s$ in $\F$.  This final contradiction proves that $\F'$ is
strong $D_s$-free.
\end{proof}

We shall also use the order-dual form of the lemma. Since $D_s$ is self-dual, applying
Lemma~\ref{lem:up-compression} to
$\{U^\perp:U\in\F\}$ gives a downward compression: if $e$ is the
maximum occupied dimension and $d=n-e$ satisfies
\eqref{eq:compression-hypotheses}, all members of dimension $e$ may be
replaced, without changing the size or creating a $D_s$, by absent
members of dimensions $e-1$ and $e-2$.

\begin{proof}[Proof of Theorem~\ref{thm:diamond-positive}]
The lower bound is immediate.  The union of any two layers contains no
chain of three distinct elements, whereas every copy of $D_s$ contains
a chain $a<b_i<c$.  Hence the union of two middle layers is $D_s$-free
and has size $\Sigmaq(n,2)$.

We first determine when the numerical hypothesis of
Lemma~\ref{lem:up-compression} is automatic.  If $2\leq s\leq q$,
then $t=q+2-s\geq2$.  Whenever $d\leq(n-2)/2$, we have
$n-d\geq d+2$, and
\begin{align*}
 2[n-d]_q
 \geq2[d+2]_q
 >(q+1)[d+1]_q,
\end{align*}
because
\[
 2[d+2]_q-(q+1)[d+1]_q
 =(q-1)[d+1]_q+2>0.
\]
Thus both hypotheses of the compression lemma hold throughout this
range.

If $s=q+1$, then $t=1$.  Whenever $d\leq(n-3)/2$, we have
$n-d\geq d+3$, and
\begin{align*}
 [n-d]_q
\geq[d+3]_q>(q+1)[d+1]_q,
\end{align*}
since
\[
 [d+3]_q-(q+1)[d+1]_q
 =(q^2-q-1)[d+1]_q+q+1>0.
\]

Now let $n=2m+1$ be odd and let $2\leq s\leq q+1$.  Starting with a
maximum strong $D_s$-free family, repeatedly apply the upward compression
while its minimum occupied dimension is less than $m$.  The relevant
integer $d$ is then at most $m-1=(n-3)/2$, so the preceding estimates
apply, including when $s=q+1$.  We obtain a maximum family whose
minimum dimension is at least $m$.  Next apply the dual downward
compression while the maximum occupied dimension is greater than
$m+1$.  If that maximum is $e\geq m+2$, then the dual minimum
dimension $n-e$ is at most $m-1$, so the same estimate applies.  The
last downward step can introduce only dimensions $m$ and $m+1$, and
therefore does not destroy the lower bound on the minimum dimension.
The resulting maximum family is contained in
 $\mathcal L_m\cup\mathcal L_{m+1}$.
Consequently its size is at most $\Sigmaq(n,2)$, proving equality for
odd $n$ throughout $2\leq s\leq q+1$.

Let $n=2m$ be even and $2\leq s\leq q$.  Repeated upward compression
while the minimum dimension is less than $m$ is valid because then
$d\leq m-1=(n-2)/2$.  It produces a maximum family with minimum
dimension at least $m$.  Repeated dual downward compression while the
maximum dimension exceeds $m+1$ is also valid: for $e\geq m+2$, the
dual minimum $n-e$ is at most $m-2$.  Such a downward step introduces
only dimensions at least $m$, so the lower bound on the minimum is
preserved.  We finish with a maximum family contained in
 $\mathcal L_m\cup\mathcal L_{m+1}$,
which proves the desired upper bound and hence equality.

Finally, consider the remaining endpoint $n=2m$ and $s=q+1$.  The
endpoint estimate is valid for every $d\leq m-2$.  We may therefore
compress upward until the minimum occupied dimension is at least
$m-1$, and then compress downward until the maximum occupied dimension
is at most $m+1$.  The downward replacements have dimensions at least
$m-1$, so the first conclusion is preserved.  This proves that some
maximum family is supported on
$\mathcal L_{m-1}\cup\mathcal L_m\cup\mathcal L_{m+1}$.
\end{proof}
For $m\geq2$, the missing central compression in the $n=2m$ case would require
\[
 (q+1)[m]_q\leq[m+1]_q=q[m]_q+1,
\]
which is false because $[m]_q>1$.  Thus this argument alone does not
reduce the even endpoint to two layers.  

\bigskip

We now turn our attention to the proof of Theorem \ref{thm:subspace-lattice-opposite-parity}. The upper bound is obtained by giving a lower bound on the number of missing subspaces. This approach was used by Kleitman \cite{Kleitman68} and others \cite{FranklKupavskii17,FranklKupavskii19} to determine the maximum number of subsets that a family $\mathcal F\subseteq 2^{[n]}$ can contain if it contains no $r$ pairwise disjoint sets.

\begin{lemma}
\label{lem:rank-prescribed-Lqd}
Let $0\leq r_0<r_1<\cdots<r_d\leq n$, and let
$\F\subseteq L_n(q)$ contain no strong copy of $L_d(q)$.  Put
\[
 x_r=\frac{|\mathcal L_r\setminus\F|}{\Gr nr}
 \qquad(0\leq r\leq n).
\]
Then
\begin{equation}
 \sum_{i=0}^d \Gr di x_{r_i}\geq1.
 \label{eq:Lqd-rank-density}
\end{equation}
\end{lemma}

\begin{proof}
Choose subspaces $A,E,X_1,\ldots,X_d\le V$ of subspaces of  $V$ whose sum is direct   and whose dimensions satisfy
\[
 \dim A=r_0,
 \qquad
 \dim E=d,
 \qquad
 \dim X_i=r_i-r_{i-1}-1
 \quad(1\leq i\leq d).
\]
Their direct sum has dimension $r_d$ and hence can be chosen inside
$V$.  For an $i$-dimensional subspace $U\leq E$, define
\[
 \varphi(U)=A\oplus X_1\oplus\cdots\oplus X_i\oplus U,
\]
where the sum over the $X_j$ is empty when $i=0$.  Then $\dim\varphi(U)
 =r_0+\sum_{j=1}^i(r_j-r_{j-1}-1)+i
 =r_i$.
Moreover, the direct-sum decomposition shows that $\varphi$ is a poset-isomorphism and
thus $\varphi$ gives a strong copy of $L_d(q)$ whose rank-$i$
elements lie in $\mathcal L_{r_i}$.

Apply a uniformly random element of $\operatorname{GL}(V)$ to this
fixed copy.  As $\F$ is strong $L_d(q)$-free, each image contains at least one subspace outside
$\F$.  The group $\operatorname{GL}(V)$ is transitive on every level,
so each one of the $\Gr di$ rank-$i$ elements of the copy has
probability $x_{r_i}$ of being outside $\F$.  Taking the expectation
of the number of missing elements in the random copy gives
\eqref{eq:Lqd-rank-density}.
\end{proof}

By the assumption on the parities of $n$ and $d$, we can write $n=2a+d-1$
and put $c_j=a+j-1$ for all $1\leq j\leq d$. Set
$S=\sum_{r=0}^{a-1}\Gr{n}{r}$.
Since $\Gr{n}{r}=\Gr{n}{n-r}$ and $n-(a+d)=a-1$, the upper outer tail $\sum_{r=a+d}^{n}\Gr{n}{r}$ has also
 size $S$ and $|L_n(q)|=\Sigmaq(n,d)+2S$.

 \begin{lemma}\label{ineq}
 Let \(q\ge3\) be a prime power, and let \(d\ge2\). With $n$, $a$, \(c_j\) and \(S\) as above, we have
 \begin{equation}
 \Gr{n}{c_j}\geq S\left(\Gr{d}{j-1}+\Gr dj\right)
 \qquad(1\leq j\leq d).
 \label{eq:Lqd-coefficient-comparison}
\end{equation}
 \end{lemma}

\begin{proof}
 For $1\leq r\leq a-1$,
\[
 \frac{\Gr{n}{r-1}}{\Gr nr}
 =\frac{q^r-1}{q^{n-r+1}-1}
 <q^{2r-n-1}
 \leq q^{-(d+2)}.
\]
Consequently,
\begin{equation}
 S<\frac{\Gr{n}{a-1}}{1-q^{-(d+2)}}.
 \label{eq:Lqd-tail-estimate}
\end{equation}

Fix $j$, and set $t=\min\{j,d+1-j\}$.
Reflection about the middle rank gives
\[
 \Gr{n}{c_j}=\Gr{n}{c_t},
 \qquad
 \Gr{d}{j-1}+\Gr{d}{j}=\Gr{d}{t-1}+\Gr dt,
\]
and $t\leq(d+1)/2$. We have
\begin{align*}
 \frac{\Gr{n}{c_t}}{\Gr{n}{a-1}}
 =\prod_{u=0}^{t-1}
   \frac{q^{n-(a+u)+1}-1}{q^{a+u}-1}>\prod_{u=0}^{t-1}q^{n-2(a+u)+1}
 =q^{\sum_{u=0}^{t-1}(d-2u)}
 =q^{t(d+1-t)}.
\end{align*}
Together with \eqref{eq:Lqd-tail-estimate}, this yields
\begin{equation}
 \frac{\Gr{n}{c_j}}S
 >\bigl(1-q^{-(d+2)}\bigr)q^{t(d+1-t)}.
 \label{eq:Lqd-central-lower}
\end{equation}

Define
\[
 \gamma_q=\prod_{h=1}^{\infty}(1-q^{-h})^{-1}.
\]
The product formula for Gaussian coefficients gives, for every
$0\leq i\leq d$,
\begin{equation}
 \Gr di<\gamma_q q^{i(d-i)}.
 \label{eq:Lqd-Gaussian-upper}
\end{equation}
Hence
\begin{align}
 \Gr{d}{t-1}+\Gr{d}{t}
 <\gamma_q\left(
 q^{(t-1)(d-t+1)}+q^{t(d-t)}\right)=\gamma_q q^{t(d+1-t)}
 \left(q^{-(d-t+1)}+q^{-t}\right)\leq\gamma_q q^{t(d+1-t)}(q^{-1}+q^{-d}).
 \label{eq:Lqd-middle-coefficient-upper}
\end{align}
Here the last inequality follows because the symmetric expression
$q^{-t}+q^{-(d+1-t)}$ is maximized at an endpoint of
$1\leq t\leq d$.

For nonnegative numbers whose sum is at most $1$, the product of
$1$ minus those numbers is at least $1$ minus their sum.  Therefore,
for $q>2$,
\[
 \prod_{h=1}^{\infty}(1-q^{-h})
 \geq1-\sum_{h=1}^{\infty}q^{-h}
 =\frac{q-2}{q-1},
\]
and hence $\gamma_q\leq\frac{q-1}{q-2}$.
The function $f(q)=\frac{q-1}{q-2}(q^{-1}+q^{-2})$ is decreasing and $g(q)=1-q^{-4}$ is increasing in $q$ for $q\ge 3$ while $f(3)<g(3)$. So if $q\geq3$ and $d\geq2$, it follows that
\begin{align*}
 \gamma_q(q^{-1}+q^{-d})
 \leq\frac{q-1}{q-2}(q^{-1}+q^{-2})<1-q^{-4}\leq1-q^{-(d+2)}.
\end{align*}  
Comparing
\eqref{eq:Lqd-central-lower} with
\eqref{eq:Lqd-middle-coefficient-upper} proves
\eqref{eq:Lqd-coefficient-comparison}.
\end{proof}

\begin{proof}[Proof of Theorem~\ref{thm:subspace-lattice-opposite-parity}]
The union of the middle $d$ layers of $L_n(q)$ is strong $L_d(q)$-free. This gives the lower bound.

For the upper bound, let $\F\subseteq L_n(q)$ be strong $L_d(q)$-free and
define the missing densities $x_r$ as in
Lemma~\ref{lem:rank-prescribed-Lqd}.  
For every $r<a$, apply \eqref{eq:Lqd-rank-density} to the sequence $r<c_1<\ldots<c_d$.
Since $\Gr d0=1$, this gives
\[
 x_r+\sum_{j=1}^d \Gr{d}{j}x_{c_j}\geq1.
\]
Multiplying by $\Gr{n}{r}$ and summing over $0\leq r<a$, we obtain
\begin{equation}
 \sum_{r=0}^{a-1}\Gr{n}{r}x_r
 +S\sum_{j=1}^d \Gr{d}{j}x_{c_j}\geq S.
 \label{eq:Lqd-lower-tail}
\end{equation}
Similarly, for every $s\geq a+d$, apply
\eqref{eq:Lqd-rank-density} to $c_1<\ldots<c_d<s$.
Using $\Gr{d}{d}=1$, multiplying by $\Gr{n}{s}$, and summing gives
\begin{equation}
 S\sum_{j=1}^d \Gr{d}{j-1}x_{c_j}
 +\sum_{s=a+d}^{n}\Gr{n}{s}x_s\geq S.
 \label{eq:Lqd-upper-tail}
\end{equation}
Adding \eqref{eq:Lqd-lower-tail} and
\eqref{eq:Lqd-upper-tail}, we find
\begin{equation}
 \sum_{r\notin\{c_1,\ldots,c_d\}}\Gr{n}{r}x_r
 +S\sum_{j=1}^d\left(\Gr{d}{j-1}+\Gr dj\right)x_{c_j}
 \geq2S.
 \label{eq:Lqd-two-tails}
\end{equation}

By Lemma~\ref{ineq}, the second sum is at most $\sum_{j=1}^d\Gr{n}{c_j}x_{c_j}$. Thus
$|L_n(q)\setminus\F|=\sum_{r=0}^{n}\Gr{n}{r}x_r\ge 2S$ and so
 $|\F|
 \leq\sum_{j=1}^d\Gr{n}{c_j}
 =\Sigmaq(n,d)$.
\end{proof}

\noindent \textbf{Remark.} Observe that the proof of Lemma \ref{ineq} shows that (\ref{eq:Lqd-coefficient-comparison}) is a strict inequality and thus (\ref{eq:Lqd-two-tails}) is a strict inequality unless $x_{c_j}$ are all zeros. This shows the uniqueness of the extremal construction of the middle $d$ layers of $L_n(q)$. 

\bigskip

\noindent \textbf{AI declaration}: The construction and proofs of this manuscript are due to ChatGPT 5.6. The authors checked and rewrote the proofs and take full responsibility for the content of the manuscript.


\begin{thebibliography}{9}
\bibitem{AMP}
    M. Axenovich, R.R. Martin, B. Patk\'os, Extremal poset theory, in A. Gagarin, R. Behrend, M. Lettington, and I. Aliev (eds.), \textit{Surveys in Combinatorics 2026}, London Mathematical Society Lecture Note Series 505, Cambridge University Press: 31--74, 2026.
\bibitem{Bukh}
B.~Bukh,
Set families with a forbidden subposet,
\textit{Electronic Journal of Combinatorics} \textbf{16}(1) (2009), Paper~R142.
\bibitem{EIL}
    D. Ellis, M.R. Ivan, I. Leader, Tur\'an densities for daisies and hypercubes. \textit{Bulletin of the London Mathematical Society}, 56(12) (2024), 3838--3853.
\bibitem{FranklKupavskii17}
P.~Frankl and A.~Kupavskii,
\newblock Families with no $s$ pairwise disjoint sets,
\newblock {\em J. London Math. Soc.} {\bf 95} (2017), 875--894.

\bibitem{FranklKupavskii19}
P.~Frankl and A.~Kupavskii,
\newblock Families of sets with no matchings of sizes $3$ and $4$,
\newblock {\em European J. Combin.} {\bf 75} (2019), 123--135.
\bibitem{Gerbner}
D. Gerbner, The covering lemma and $q$-analogues of extremal set theory problems, \textit{Ars Mathematica Contemporanea} 24(1) 2024, Paper No~1.07.
\bibitem{GP}
D. Gerbner, B. Patk\'os, \textit{Extremal Finite Set Theory}, Chapman and Hall/CRC, Boca Raton, 2018.
\bibitem{GerbnerPatkos}
D. Gerbner, B. Patk\'os, A note on vertex Tur\'an problems in the Kneser cube, \textit{Graphs Combin.} \textbf{42} (2026), Paper No. 17, 11 pp.
\bibitem{GriggsLu}
J.R. Griggs, L. Lu, On families of subsets with a forbidden subposet, \textit{Combin. Probab. Comput.} \textbf{18} (2009), pp. 731--748.
\bibitem{GLsurv}
     J.R. Griggs, W.T. Li, Progress on poset-free families of subsets. In Recent trends in combinatorics (2016) pp. 317-338. Cham: Springer International Publishing.

\bibitem{KatonaTarjan}
G. O. H. Katona, T. G. Tarj\'an, Extremal problems with excluded subgraphs in the $n$-cube, in \textit{Graph Theory}, \L{}ag\'ow, 1981, Lecture Notes in Math., vol. 1018, Springer, Berlin, 1983, pp. 84--93.

\bibitem{Kleitman68}
D.~J.~Kleitman,
\newblock Maximal number of subsets of a finite set no $k$ of which are
pairwise disjoint,
\newblock {\em J. Combinatorial Theory} {\bf 5} (1968), 157--163.

\bibitem{P2}
B. Patk\'os, Induced and Non-induced Forbidden Subposet Problems, \textit{The Electronic Journal of Combinatorics}, (2015) P1-30.

\bibitem{Patkos}
B. Patk\'os, On the maximum size of $B_3$-free and $D_s$-free families, arXiv preprint, 2607.11753

\bibitem{PT}
B. Patk\'os, C. Tompkins, Crown-free families and forbidden subposets with $e(P)\in\{1,2\}$, arXiv preprint, 2608.17580
\bibitem{SS}
G. Sarkis, S. Shahriari and PCURC, Diamond-free subsets in the linear lattices, \textit{Order},
31 (2014), 421--433.
\bibitem{SY}
S. Shahriari, S. Yu, Avoiding brooms, forks, and butterflies in the linear lattices. \textit{Order}, 37(2) (2020), 223--242.

\bibitem{Tompkins}
C. Tompkins, An Improved Lower Bound for Diamond-Free Families, arXiv preprint, 2607.09497

\bibitem{XT}
J. Xiao, C. Tompkins, On forbidden poset problems in the linear lattice, \textit{Electronic Journal of Combinatorics}, 27(1) 2020, P1.18. 


\end{thebibliography}
\end{document}